\documentclass[12pt]{article}
\usepackage{amsmath,amssymb,amsbsy,amsfonts,amsthm,latexsym,amsopn,amstext,amsxtra,euscript,amscd,a4wide}

\newfont{\teneufm}{eufm10}
\newfont{\seveneufm}{eufm7}
\newfont{\fiveeufm}{eufm5}
\newfam\eufmfam
                 \textfont\eufmfam=\teneufm \scriptfont\eufmfam=\seveneufm
                 \scriptscriptfont\eufmfam=\fiveeufm
\def\bbbc{{\mathchoice {\setbox0=\hbox{$\displaystyle\rm C$}\hbox{\hbox
to0pt{\kern0.4\wd0\vrule height0.9\ht0\hss}\box0}}
{\setbox0=\hbox{$\textstyle\rm C$}\hbox{\hbox
to0pt{\kern0.4\wd0\vrule height0.9\ht0\hss}\box0}}
{\setbox0=\hbox{$\scriptstyle\rm C$}\hbox{\hbox
to0pt{\kern0.4\wd0\vrule height0.9\ht0\hss}\box0}}
{\setbox0=\hbox{$\scriptscriptstyle\rm C$}\hbox{\hbox
to0pt{\kern0.4\wd0\vrule height0.9\ht0\hss}\box0}}}}
\def\bbbq{{\mathchoice {\setbox0=\hbox{$\displaystyle\rm
Q$}\hbox{\raise
0.15\ht0\hbox to0pt{\kern0.4\wd0\vrule height0.8\ht0\hss}\box0}}
{\setbox0=\hbox{$\textstyle\rm Q$}\hbox{\raise
0.15\ht0\hbox to0pt{\kern0.4\wd0\vrule height0.8\ht0\hss}\box0}}
{\setbox0=\hbox{$\scriptstyle\rm Q$}\hbox{\raise
0.15\ht0\hbox to0pt{\kern0.4\wd0\vrule height0.7\ht0\hss}\box0}}
{\setbox0=\hbox{$\scriptscriptstyle\rm Q$}\hbox{\raise
0.15\ht0\hbox to0pt{\kern0.4\wd0\vrule height0.7\ht0\hss}\box0}}}}
\def\bbbt{{\mathchoice {\setbox0=\hbox{$\displaystyle\rm
T$}\hbox{\hbox to0pt{\kern0.3\wd0\vrule height0.9\ht0\hss}\box0}}
{\setbox0=\hbox{$\textstyle\rm T$}\hbox{\hbox
to0pt{\kern0.3\wd0\vrule height0.9\ht0\hss}\box0}}
{\setbox0=\hbox{$\scriptstyle\rm T$}\hbox{\hbox
to0pt{\kern0.3\wd0\vrule height0.9\ht0\hss}\box0}}
{\setbox0=\hbox{$\scriptscriptstyle\rm T$}\hbox{\hbox
to0pt{\kern0.3\wd0\vrule height0.9\ht0\hss}\box0}}}}
\def\bbbs{{\mathchoice
{\setbox0=\hbox{$\displaystyle     \rm S$}\hbox{\raise0.5\ht0\hbox
to0pt{\kern0.35\wd0\vrule height0.45\ht0\hss}\hbox
to0pt{\kern0.55\wd0\vrule height0.5\ht0\hss}\box0}}
{\setbox0=\hbox{$\textstyle        \rm S$}\hbox{\raise0.5\ht0\hbox
to0pt{\kern0.35\wd0\vrule height0.45\ht0\hss}\hbox
to0pt{\kern0.55\wd0\vrule height0.5\ht0\hss}\box0}}
{\setbox0=\hbox{$\scriptstyle      \rm S$}\hbox{\raise0.5\ht0\hbox
to0pt{\kern0.35\wd0\vrule height0.45\ht0\hss}\raise0.05\ht0\hbox
to0pt{\kern0.5\wd0\vrule height0.45\ht0\hss}\box0}}
{\setbox0=\hbox{$\scriptscriptstyle\rm S$}\hbox{\raise0.5\ht0\hbox
to0pt{\kern0.4\wd0\vrule height0.45\ht0\hss}\raise0.05\ht0\hbox
to0pt{\kern0.55\wd0\vrule height0.45\ht0\hss}\box0}}}}
\def\bbbz{{\mathchoice {\hbox{$\sf\textstyle Z\kern-0.4em Z$}}
{\hbox{$\sf\textstyle Z\kern-0.4em Z$}}
{\hbox{$\sf\scriptstyle Z\kern-0.3em Z$}}
{\hbox{$\sf\scriptscriptstyle Z\kern-0.2em Z$}}}}

\newtheorem{theorem}{Theorem}
\newtheorem{lemma}[theorem]{Lemma}

\def\cA{{\mathcal A}}
\def\cB{{\mathcal B}}

\def\cI{{\mathcal I}}
\def\cJ{{\mathcal J}}

\def\cS{{\mathcal S}}

\def\cZ{{\mathcal Z}}

\def\({\left(}
\def\){\right)}
\def\[{\left[}
\def\]{\right]}
\def\<{\langle}
\def\>{\rangle}

\def\fl#1{\left\lfloor#1\right\rfloor}

\def\F{\mathbb{F}}

\def\mand{\qquad\mbox{and}\qquad}

\newcommand{\set}[1]{\left\{#1\right\}}

\newcommand{\To}{\longrightarrow}

\newcommand{\D}{\mathcal{D}}

\newcommand{\E}{\mathbf{E}}

\newcommand{\LL}{\mathrm{L}}

\newcommand{\HH}{\mathrm{H}}

\newcommand{\W}{\mathrm{W}}
\newcommand{\BW}{\mathrm{BW}}
\newcommand{\JQ}{\mathrm{JQ}}
\newcommand{\JI}{\mathrm{JI}}

\newcommand{\GH}{\mathrm{GH}}

\newcommand{\bx} X
\newcommand{\by} Y
\newcommand{\bz} Z

\def\wX{\widetilde{X}}
\def\wY{\widetilde{Y}}
\def\wa{\widetilde{a}}
\def\wE{\widetilde{E}}

\def\bF{\overline{\F}}

\def\Tr{\mathrm{Tr}}

\def\JL{J_{\mathrm{L}}(q)}
\def\IL{I_{\mathrm{L}}(q)}

\def\JH{J_{\mathrm{H}}(q)}
\def\JJQ{J_{\mathrm{JQ}}(q)}
\def\JJI{J_{\mathrm{JI}}(q)}

\def\IH{I_{\mathrm{H}}(q)}
\def\IJQ{I_{\mathrm{JQ}}(q)}
\def\IJI{I_{\mathrm{JI}}(q)}

\def\JGH{J_{\mathrm{GH}}(q)}
\def\SJGH{\cJ_{\mathrm{GH}}}
\def\JGHv{\cJ_{\mathrm{GH}, v}}
\def\IGH{I_{\mathrm{GH}}(q)}

\newcommand{\keywords}[1]{{
      \list{}{\advance\topsep by -5ex \relax\small \leftmargin=1cm
      \labelwidth=0pt \listparindent=0pt \itemindent\listparindent
      \rightmargin\leftmargin}\item[\hskip\labelsep \bfseries
      Keywords:] {#1} \endlist}}

\begin{document}

\pagestyle {plain}
\pagenumbering{arabic}

\title{On the Number of Distinct Legendre, Jacobi and Hessian Curves}
\author{
         {\sc{Reza Rezaeian Farashahi}} \\
         {Department of Computing}\\
         {Macquarie University} \\
         {Sydney, NSW 2109, Australia} \\
         {\tt reza@ics.mq.edu.au}}

\date{}
\maketitle

\begin{abstract}
We give explicit formulas for the number of distinct elliptic curves over a finite field, up to
isomorphism, in the families of Legendre, Jacobi, Hessian and generalized Hessian curves.
\end{abstract}




\section{Introduction}
\label{sec intro}

A nonsingular absolutely irreducible projective curve of genus~1 defined over a field $\F$ with at
least one $\F$-rational point is called an elliptic curve over $\F$,  see~\cite{ACDFLNV,Silv} for a
general background on elliptic curves. Koblitz \cite{Ko} and Miller \cite{Mi} were the first to
show that the group of rational points on an elliptic curve over a finite field can be used for the
discrete logarithm problem in a public-key cryptosystem.

In particular, an elliptic curve $E$ over $\F$ can be given by the so-called {\it Weierstrass
equation\/}
\begin{equation}
\label{eq:Weier} E :\quad  \by^2 + a_1\bx\by + a_3\by = \bx^3 + a_2\bx^2 + a_4\bx + a_6,
\end{equation}
where the coefficients $a_1, a_2, a_3, a_4, a_6\in \F$, which traditionally has been most commonly
used. For a field $\F$ of characteristic $p\ne 2,3$, the Weierstrass equation \eqref{eq:Weier} can
be transformed to the so called {\it short Weierstrass equation\/} given by
\begin{equation}
\label{eq:sWeier} \E_{\W,u,v} :\quad  \by^2 = \bx^3 + u\bx + v,
\end{equation}
where the coefficients $u, v\in \F$ (see~\cite{ACDFLNV,Silv}).

There are many other forms of equations to represent elliptic curves such as Legendre equation,
Hessian equation, quartic equation and intersection of two quadratic surfaces~\cite[Chapter
2]{Wash}. Any equation given by each of the latter forms over a field $\F$ can be transformed to a
Weierstrass equation by change of variables that uses rational functions with coefficient in $\F$.
Usually, any elliptic curve over an algebraically closed field $\F$ can be defined by each of the
latter equations.

A {\it Legendre equation} is a variant of Weierstrass equation with one parameter. Any elliptic
curve defined over an algebraically closed field $\F$ of characteristic $p\ne 2$ can be expressed
by an elliptic curve by the Legendre equation
\begin{equation}
\label{eq:Leg} \E_{\mathrm{L},u}: \quad  Y^2=X(X-1)(X-u),
\end{equation}
for some $u \in \F$. Furthermore, an elliptic curve in Legendre form is birationally equivalent to
a {\it Jacobi quartic} curve that is given by the equation
\begin{equation}
\label{eq:JacQ} \E_{\JQ,u}: \quad  Y^2 =X^4+2uX^2+1.
\end{equation}
for some $u\in \F$ with $u\ne \pm1$. Moreover, an elliptic curve in Legendre form is birationally
equivalent to a so-called {\it Jacobi intersection} that is defined by the intersection of two
quadratics given by
\begin{equation}
\label{eq:JacI} \E_{\JI,u}: \quad  X^2+Y^2=1 \mand  uX^2+Z^2=1,
\end{equation}
where $u\in \F$ and $u\ne 0,1$. See~\cite{Cas} for more background on Jacobi curves. The latter two
forms over finite fields are used for cryptographic interest in~\cite{LiSm}. Also, for the recent
improvements on their arithmetic see~\cite{Du,HWCD2}.

A {\it Hessian curve} over a field $\F$ is given by the cubic equation
\begin{equation}
\label{eq:Hes} \E_{\HH,u}: \quad  X^3 +Y^3+1 = uXY,
\end{equation}
for some  $u \in \F$ with $u^3\ne 27$ (see~\cite{Cas}). For the cryptographic interests on Hessian
curves over finite fields see~\cite{BKL,FJ,HWCD2,JoQu,Sma}. Recently, Farashahi and Joye have
considered the generalization of Hessian curves to the so-called {\it generalized Hessian} family,
\cite{FJ}, given by
\begin{equation}
\label{eq:GH} \E_{\GH,u,v}: \quad  \bx^3 +\by^3 + v  = u\bx\by,
\end{equation}
where $u, v \in \F$,  $v\ne 0$ and $u^3\ne 27v$. Moreover, Bernstein, Kohel and Lange, \cite{BKL},
have also considered the \emph{twisted Hessian} form that is similar to the latter form up to the
order of the coordinates.

%

We note that, above families do not cover all distinct curves over finite fields. Accordingly, a
natural question arise about the number of isomorphism classes of these curves.

Lenstra~\cite{Le} gave explicit estimates for the number of isomorphism classes of elliptic curves
over a prime field $\F_p$ with order divisible by a prime $l\ne p$. After that, Howe~\cite{Ho}
extended Lenstra's work to arbitrary integers $l$. Moreover, Castryck and Hubrechts~\cite{CaHu}
generalized these results giving explicit estimates for the number of isomorphism classes of
elliptic curves over a finite field $\F_q$ having order with a fix remainder divided by an
integer~$l$.

Furthermore, Farashahi and Shparlinski \cite{FS}, using the notion of the {\it $j$-invariant of an
elliptic curve\/}, see~\cite{ACDFLNV,Silv,Wash}, gave exact formulas for the number of distinct
elliptic curves over a finite field (up to isomorphism over the algebraic closure of the ground
field) in the families of Edwards curves~\cite{Edw} and  their generalization due to Bernstein and
Lange~\cite{BerLan1} as well as the curves introduced by Doche, Icart and Kohel~\cite{DIK}.
Moreover, the open question of \cite{FS} is whether there are explicit formulas for the number of
distinct elliptic curves over a finite field in the families of Hessian curves, Jacobi quartic and
Jacobi intersections.

In this paper, we give precise formulas for the number of distinct $j$-invariants of elliptic
curves over a finite field in the families of Legendre, Jacobi and Hessian curves. The next
interesting and more challenging step is to study isomorphism classes over the ground field of
these families. Moreover, we give exact formulas for the number of isomorphism classes over the
ground field of above families.

Throughout the paper, for a field $\F$, we denote its algebraic closure by $\overline{\F}$ and its
multiplicative subgroup by $\F^*$. The letter $p$ always denotes a prime number and the letter $q$
always denotes a prime power. As usual, $\F_q$ is a finite field of size $q$. Let $\chi_2$ denote
the quadratic character in $\F_q$, where $p\ge 3$. So, for any $q$ where $p\ge 3$, $u=w^2$ for some
$w \in \F^*_q$ if and only if $\chi_2(u) =1$. The cardinality of a finite set $\cS$ is denoted by
$\#\cS$.

\section{Background on isomorphisms and outline of our approach}

\label{sec:iso}

An elliptic curve $E$ over $\F$ given by the Weierstrass equation~\eqref{eq:Weier} can be
transformed to the elliptic curve $\widetilde{E}$ over $\F$ given by the Weierstrass equation
$$
\widetilde{E} :\quad  \wY^2 + \wa_1\wX\wY + \wa_3\wY = \wX^3 + \wa_2\wX^2 + \wa_4\wX + \wa_6,
$$
via the invertible maps $\, X \mapsto \alpha^2\wX+ \beta$ and $\, Y \mapsto \alpha^3\wY+
\alpha^2\gamma \wX+ \delta $ with $\alpha, \beta, \gamma, \delta \in \bF$ and $\alpha \ne 0$.  In
this case, the elliptic curves $E$ and $\wE$ are called {\it isomorphic} over $\bF$ or {\it twists}
of each other. In case $\alpha, \beta, \gamma, \delta \in \F$, the elliptic curves $E$ and $\wE$
are called {\it isomorphic} over $\F$. We use $E \cong_{\F} \wE$ to denote $E$ and $\wE$ are
$\F$-{\it isomorphic}.

The elliptic curve $E$ over $\F$ given by the Weierstrass equation~\eqref{eq:Weier} has the
non-zero discriminant
$$\Delta_E=-b_2^2b_8 -8b_4^3-27b_6^2+ 9b_2b_4b_6,$$
\begin{equation*}
b_2 = a_1^2+ 4a_2,\ b_4 = a_1a_3 + 2a_4,\ b_6 = a_3^2+ 4a_6,\ b_8 = a_1^2a_6 - a_1a_3a_4 + 4a_2a_6
+ a_2a_3^2- a_4^2.
\end{equation*}
Also, its $j$-invariant is explicitly defined as
$$
j(E) = (b_2^2-24b_4)^3/\Delta_E.
$$
It is known that two elliptic curves $E, \wE$ over a field $\F$ are isomorphic over $\overline{\F}$
if and only if $j(E_1) = j(E_2)$, see~\cite[Proposition~III.1.4(b)]{Silv}.

Note that over a finite field $\F_q$ each of the $q$ values appears as $j$-invariant of at least
one curve. So, the number of distinct elliptic curves over~$\F_q$ (up to isomorphism over
$\overline{\F}_q$) is equal to $q$. The same is true for the family~\eqref{eq:sWeier}. Furthermore,
for a finite field $\F_q$ with characteristic $p\ne 2,3$, the number of $\F_q$-isomorphism classes
of the family~\eqref{eq:sWeier} is equal to $2q+6$, $2q+2$, $2q+4$, $2q$ if $q \equiv 1, 5, 7, 11
\pmod {12}$ respectively (e.g. see~\cite{Le}).

In the following, we use $\JL$, $\JJQ$, $\JJI$, $\JH$ and $\JGH$
to denote the number of distinct $j$-invariants of the curves defined over $\F_q$ in the
families~\eqref{eq:Leg}, \eqref{eq:JacQ}, \eqref{eq:JacI}, \eqref{eq:Hes} and~\eqref{eq:GH},
respectively.

Moreover, we use $\IL$, $\IJQ$, $\IJI$, $\IH$ and $\IGH$
to denote the number of $\F_q$-isomorphism classes of the families~\eqref{eq:Leg}, \eqref{eq:JacQ},
\eqref{eq:JacI}, \eqref{eq:Hes} and~\eqref{eq:GH}
respectively.

We compute the number of distinct $j$-invariants of a family of elliptic curves $\E_u$ over a
finite field $\F_q$ with a parameter $u$, using the general approach mentioned in~\cite{FS}. In
this approach, the $j$-invariant of $\E_u$ is given by a rational function $F(U)\in \F_q(U)$ of
small degree. Next, we consider the bivariate rational function
$$
F(U) - F(V) = {g(U,V)}/{l(U,V)}
$$
with two relatively prime polynomials $g$ and $l$. We factor $g(U,V)$. Then, studying the number of
distinct roots of the polynomials $g_u(V)=g(u,V)$, for $u\in \F_q$, provides the necessary
information, which is the cardinality of the set $\cJ_u$ of all curves $\E_v$ with
$j(\E_v)=j(\E_u)$. Then, for several small integers $k$, we count the number of elements $u$ of
$\F_q$ with $\#\cJ_u=k$. Therefore, we obtain the number of different sets $\cJ_u$, i.e., the
number of distinct $j(\E_u)$ in the family.

We propose an analogous approach to count the number of $\F_q$-isomorphism classes of elliptic
curves $\E_u$, for $u\in \F_q$. Considering the set $\cJ_u$, we study the set $\cI_u$ of curves
$\E_v$ which are $\F_q$-isomorphic to the curve $\E_u$. Then, counting the number of distinct
sets~$\cI_u$ provides our results.

\section{Legendre curves}

\label{subsec:Lpre}

We consider the curves $\E_{\LL,u}$ given by Legendre equation~\eqref{eq:Leg} over a finite
field~$\F_q$ with characteristic $p\ge 3$. We note that $u \ne 0,1$, since the curve $\E_{\LL,u}$
is nonsingular. The Legendre curve $\E_{\LL,u}$ over $\F_q$ of characteristic $p > 3$ is isomorphic
to the Weierstrass curve $\E_{\W,a_u,b_u}$ given by
\begin{equation}
\label{eq:LW} \by^2 = \bx^3 + a_u \bx + b_u,
\end{equation}
where
$$a_u=-\(u^2-u+1\)/3,\quad b_u=-(u+1)(u-2)(2u-1)/27.$$ The $j$-invariant of
$\E_{\LL,u}$ is given by $j(\E_{\LL,u}) = F(u)$ where
\begin{equation*}
F(U) = \frac{2^8(U^2-U+1)^3}{(U^2-U)^2}.
\end{equation*}

Here, we study the cardinality of preimages of $F(u)$, for all elements $u\in \F_q \setminus
\set{0,1}$, under the map $u \mapsto F(u)$. In particular, we see this map is~$6:1$, for almost
all~$u\in \F_q$.

We consider the bivariate rational function $ F(U) - F(V) = {g(U,V)}/{l(U,V)} $ with two relatively
prime polynomials $g$ and $l$. We see that
\begin{equation*}
\begin{aligned}
\!g(U,V )\!=\!2^8 (U\!-\!V)(U\!+\! V\!\! -\! 1)(UV\!\! -\! 1)(UV\!\! -\! V \!+\! 1)(UV \!\!-\! U
+\! 1)(UV\!\! -\! U\! -\! V).
\end{aligned}
\end{equation*}
Then, we need to study the number of roots of the polynomial $g_u(V)=g(u,V)$, for $u\in \F_q
\setminus \set{0,1}$. Moreover, for $u\in \F_q \setminus\set{0,1},$ we let
\begin{eqnarray*}
\cJ_{\LL,u}=\set{v~:~v \in \F_q, \ \ \E_{\LL,u} \cong_{\overline{\F}_q} \E_{\LL,v}}.
\end{eqnarray*}
We note that, for all $v\in \cJ_{\LL,u}$, the curves $\E_{\LL,u}$ and $\E_{\LL,v}$ have the same
$j$-invariants. But, these curves may not be isomorphic over $\F_q$. Next, for a fixed value $u\in
\F_q \setminus \set{0,1}$, we let
$$
\cI_{\LL,u}=\set{v~:~v \in \F_q, \ \E_{\LL,u} \cong_{\F_q} \E_{\LL,v}}.
$$
Clearly, for all $u \in \F_q\setminus\set{0,1}$, we have $\cI_{\LL,u} \subseteq \cJ_{\LL,u}$. We
see that $j(\E_{\LL,u})=j(\E_{\W,a_u,b_u})=0$ if and only if $a_u=0$. Furthermore, $j(\E_{\LL,u})=
1728$ if and only if $b_u=0$ (see Equation~\eqref{eq:LW}). Let
\begin{equation}\label{eq:BL}
\cB=\set {u: u\in \F_q, a_u b_u=0}.
\end{equation}
The following lemma gives the cardinality of $\cJ_{\LL,u}$, for all $u\in \F_q\setminus\set{0,1}$.

\begin{lemma}\label{lem:JuL}
For all  $u\in \F_q\setminus \set{0,1}$, we have
$$
\#\cJ_{\LL,u}= \left\{
\begin{array}{ll}
1, & \text{ if }  u = -1 \text{ and } p=3,\\
3, & \text{ if }  u \in \set{-1, 2, 2^{-1}} \text{ and } p>3,\\
2, & \text{ if } u^2-u+1=0 \text{ and } p>3,\\
6, & \text{ if } u \notin \cB.
\end{array}
\right.
$$
\end{lemma}

\begin{proof}
For a fixed value $u\in \F_q \setminus \set{0,1}$, let $g_u(V)=g(u,V)$ and
$$
\cZ_u=\set{v~:~v \in \F_q\setminus\set{0,1}, \ g_u(v)=0}.
$$ Then,
%
$$
\cZ_u=\set{ \textstyle u, \frac{1}{u}, 1-u, \frac{1}{1-u},  \frac{u-1}{u}, \frac{u}{u-1}}.
$$
We note that, for all $v\in \cJ_{\LL,u}$, the curves $\E_{\LL,u}$ and $\E_{\LL,v}$ have the same
$j$-invariants, so $\cJ_{\LL,u}=\cZ_u$.
Next, we consider several cases depending on the value of $u$ in $\F_q$.
\begin{itemize}
\item
If $u=-1$ and $p=3$, then $\cJ_{\LL,-1}=\set{-1}$.
\item
If $u \in \set{-1,2,2^{-1}}$ and $p>3$, then $\cJ_{\LL,u}=\set{-1, 2, 2^{-1}}$.
\item
If $u^2-u+1=0$ and $p>3$, then $\cJ_{\LL,u}=\set{u, \frac{1}{u}}$.
\item
If $u\ne -1$, $u\ne 2$,  $u\ne 2^{-1}$ and $u^2-u+1\ne 0$, then all 6 elements in $\cJ_{\LL,u}$ are
distinct.
\end{itemize}
So, the proof of the lemma ic complete.
\end{proof}

Let $\chi_2$ denote the quadratic character in $\F_q$, where $p\ge 3$. So, for any $q$ where $p\ge
3$, $u=w^2$ for some $w \in \F^*_q$ if and only if $\chi_2(u) =1$.

\begin{lemma}\label{lem:isoL}
For all elements $u,v \in \F_q\setminus \set{0,1}$, we have $\E_{\LL,u} \cong_{\F_q} \E_{\LL,v}$ if
and only if $u$, $v$ satisfy one of the following:

\begin{enumerate}
  \item $v=u,$
  \item $v=\frac{1}{u}$ and $\chi_2(u)=1$
  \item $v=1-u$ and $\chi_2(-1)=1$
  \item $v=\frac{1}{1-u}$ and $\chi_2(u-1)=1$
  \item $v=\frac{u-1}{u}$ and $\chi_2(-u)=1$
  \item $v=\frac{u}{u-1}$ and $\chi_2(1-u)=1$.
\end{enumerate}
\end{lemma}
\begin{proof}
For $u,v \in \F_q\setminus \set{0,1}$, we have $\E_{\LL,u} \cong_{\F_q} \E_{\LL,v}$ if and only if
there exist elements $\alpha, \beta$ in $\F_q$, where $\alpha\ne 0$ and
$$\set{\frac{-\beta}{\alpha^2}, \frac{1-\beta}{\alpha^2}, \frac{u-\beta}{\alpha^2}}=\set{0,1,v}.$$
This is equivalent to one of the following cases:
\begin{itemize}
\item $\beta=0, \alpha^2=1$ and $v=u$,
\item $\beta=0, \alpha^2=u$ and $v=\frac{1}{u}$,
\item $\beta=1, \alpha^2=-1$ and $v=1-u$,
\item $\beta=1, \alpha^2=u-1$ and $v=\frac{1}{1-u}$,
\item $\beta=u, \alpha^2=-u$ and $v=\frac{u-1}{u}$,
\item $\beta=u, \alpha^2=1-u$ and $v=\frac{u}{u-1}$,
\end{itemize}
which concludes the proof.
\end{proof}

Furthermore, the following lemma gives the cardinality of $\cI_{\LL,u}$, for all $u\in
\F_q\setminus \set{0,1}$.

\begin{lemma}\label{lem:IuL}
For all  $u\in \F_q\setminus \set{0,1}$, we have
$$
\#\cI_{\LL,u}= \left\{
\begin{array}{ll}
1, & \text{ if }  u = -1 \text{ and } p=3,\\
3, & \text{ if }  u \in \set{-1, 2, 2^{-1}}, \ q \equiv 1,3,7  \pmod 8 \text{ and } p>3,\\
2, & \text{ if }  u \in\set{-1, 2}, \ q \equiv 5 \pmod  8,\\
1, & \text{ if }  u = 2^{-1}, \ q \equiv 5 \pmod 8,\\
2, & \text{ if }  u^2-u+1=0,\ q \equiv 1 \pmod {12}  \text{ and } p>3,\\
1, & \text{ if }  u^2-u+1=0,\ q \not\equiv 1 \pmod {12},\\
3, & \text{ if }  \chi_2(-1)=-1  \text{ and } u \notin \cB,\\
2, & \text{ if }  \chi_2(-1)=1, \chi_2(u)=\chi_2(1-u)=-1  \text{ and }\ u \notin \cB,\\
4, & \text{ if }  \chi_2(-1)=1, \chi_2(u)\chi_2(1-u)=-1  \text{ and }\ u \notin \cB,\\
6, & \text{ if }  \chi_2(-1)=1, \chi_2(u)=\chi_2(1-u)=1  \text{ and }\ u \notin \cB.\\
\end{array}
\right.
$$
\end{lemma}

\begin{proof}
For $p=3$, we have $\cB=\set{-1}$. Then, from Lemma~\ref{lem:JuL}, we obtain
$$\cI_{\LL,-1}=\set{-1},\ \text{ if }  p=3.  $$

Next, we assume that $p>3$. From Lemma~\ref{lem:isoL} (part $5$), we see $\E_{\LL,-1}\cong_{\F_q}
\E_{\LL,2}$. Furthermore, $\E_{\LL,-1}\cong_{\F_q} \E_{\LL,2^{-1}}$ if and only if $\chi_2(2)=1$ or
$\chi_2(-2)=1$, which is equivalent to the cases where $q \equiv 1,3,7 \pmod 8$. From the proof of
Lemma~\ref{lem:JuL}, if $u \in \set{-1, 2, 2^{-1}}$, we have $\cJ_{\LL,u}=\set{-1,2,2^{-1}}$.
Therefore,
$$
\cI_{\LL,u}= \left\{
\begin{array}{ll}
\set{-1,2,2^{-1}}, & \text{ if }  u \in \set{-1, 2, 2^{-1}}, \ q \equiv 1,3,7 \pmod 8,\\
\set{-1,2}, & \text{ if }  u \in \set{-1, 2}, \ q \equiv 5 \pmod 8,\\
\set{2^{-1}}, & \text{ if }  u = 2^{-1}, \ q \equiv 5 \pmod 8.
\end{array}
\right.
$$
Now, we assume that $u\in \F_q$ with $u^2-u+1=0$. This happens if $\chi_2(-3)=1$. Then,
$u=\frac{1+\zeta}{2}$, where $\zeta$ is a square root of $-3$ in $\F_q$. Furthermore, $u$ can be
written as $u=-(\frac{1-\zeta}{2})^2$.  Then, from Lemma~\ref{lem:isoL}, we see that
$\E_{\LL,u}\cong_{\F_q} \E_{\LL,\frac{1}{u}}$ if and only if $\chi_2(-1)=1$. Moreover,
$\chi_2(-3)=\chi_2(-1)=1$ if and only if $q \equiv 1 \; (\text{mod } 12)$. From the proof of
Lemma~\ref{lem:JuL}, we have $\cJ_{\LL,u}=\set{u,\frac{1}{u}}$. Hence,
$$
\cI_{\LL,u}= \left\{
\begin{array}{ll}
\set{u, \frac{1}{u}}, & \text{ if }  u^2-u+1=0,\ q \equiv 1 \pmod {12},\\
\set{u},              & \text{ if }  u^2-u+1=0,\ q \not\equiv 1 \pmod {12}.\\
\end{array}
\right.
$$

From now on, we let $u\in \F_q\setminus \set{0,1}$ with $u \notin \cB$. By the the proof of
Lemma~\ref{lem:JuL}, we have $$\cJ_{\LL,u}=\set{u, \frac{1}{u}, 1-u, \frac{1}{1-u}, \frac{u-1}{u},
\frac{u}{u-1}}.$$ We distinguish the following cases.
\begin{itemize}
\item
First, we assume that $\chi_2(-1)=-1$. From Lemma~\ref{lem:isoL}, we see that exactly one of
$\frac{1}{u}$ and $\frac{u-1}{u}$ belongs to $\cI_{\LL,u}$. Also, exactly one of $\frac{1}{1-u}$
and $\frac{u}{u-1}$ belongs to $\cI_{\LL,u}$. Furthermore, $1-u \notin \cI_{\LL,u}$. Therefore, if
$\chi_2(-1)=-1$, we obtain $\#\cI_{\LL,u}=3$.
\item
Now, we assume that $\chi_2(-1)=1$. So, $1-u\in \cI_{\LL,u}$.
\begin{itemize}
  \item
  If $\chi_2(u)=\chi_2(1-u)=-1$, then $\chi_2(-u)=\chi_2(u-1)=-1$. Next, from~Lemma~\ref{lem:isoL}, we have $\#\cI_{\LL,u}=2$.
  \item
  If $\chi_2(u)\chi_2(1-u)=-1$, then $\chi_2(-u)\chi_2(u-1)=-1$. So, we have $\cI_{\LL,u}=\set{1-u, \frac{1}{u}, \frac{u-1}{u}}$ if $\chi_2(u)=1$
  and  $\cI_{\LL,u}=\set{1-u, \frac{1}{1-u},\frac{u}{u-1}}$ if $\chi_2(u)=-1$. Thus, $\#\cI_{\LL,u}=4$.
  \item
  If $\chi_2(u)=\chi_2(1-u)=1$, then $\chi_2(-u)\chi_2(u-1)=1$. Hence, $\cI_{\LL,u}=\cJ_{\LL,u}$ (see Lemma~\ref{lem:isoL}).
\end{itemize}
\end{itemize}
So, the proof of this lemma is complete.
\end{proof}

The following Lemma is used in the proof of Theorem~\ref{thm:numi Leg}, which can be of independent
interest.

\begin{lemma}
\label{lem:Lchi} Let $\chi_2$ be the quadratic character of a finite field $\F_q$ of characteristic
$p\ne 2$. For $i,j \in \set{-1,1},$ let
$$S_{i,j}=\set{u: u\in \F_q, \chi_2(u)=i, \chi_2(1-u)=j }.$$
We have
$$\#S_{i,j}=
\left\{
\begin{array}{ll}
(q-5)/4,&   \text{ if }  (i,j) = (1,1) \text{ and } q \equiv 1 \pmod {4}, \vspace{2pt}\\
(q-1)/4,&   \text{ if }  (i,j) \ne (1,1) \text{ and } q \equiv 1 \pmod {4}, \vspace{2pt}\\
(q-3)/4,&   \text{ if }  (i,j)\ne (-1,-1) \text{ and } q \equiv 3 \pmod {4}, \vspace{2pt}\\
(q+1)/4,&  \text{ if }  (i,j)=(-1,-1) \text{ and } q \equiv 3 \pmod {4}.\\
\end{array}
\right.
$$
\end{lemma}
\begin{proof}
Consider the set $S_{i,j}$, for fixed $i,j$ in $\set{-1,1}$. Let $C_{i,j}$ be an affine conic over
$\F_q$ given by the equation
$$C_{i,j}: \alpha X^2+ \beta Y^2 =1, $$
where $\chi_2(\alpha)=i$, $\chi_2(\beta)=j$. Let $C_{i,j}(\F_q)$ be the set of affine
$\F_q$-rational points on $C_{i,j}$. We note that, $\#C_{i,j}(\F_q)=q-1$ if $\chi_2(-1)=ij$ and
$\#C_{i,j}(\F_q)=q+1$ if $\chi_2(-1)=-ij$. Let $T_{i,j}= C_{i,j}(\F_q) \setminus \set{(x,y) \in
C_{i,j}(\F_q)~:~xy=0}$. Then, we see
$$\#T_{i,j}=
\left\{
\begin{array}{ll}
q-5, &  \text{ if }  \chi_2(-1)=1, \ i=j=1, \  \\
q-1, &  \text{ if }  \chi_2(-1)=1, \ i=-1  \text{ or } j=-1,  \\
q-3, &  \text{ if }  \chi_2(-1)=-1, \ i=1  \text{ or } j=1,  \\
q+1, &  \text{ if }  \chi_2(-1)=-1, \ i=j=-1. \\
\end{array}
\right.
$$
We note that $\chi_2(-1)=1$ if and only if $q \equiv 1 \pmod {4}$. Next, we consider the map $\tau
: T_{i,j} \To S_{i,j}$ given by $x \mapsto \alpha x^2 $. The map $\tau$ is surjective. Moreover, it
is a $4 : 1$ map. So, $\#S_{i,j}=\frac{\#T_{i,j}}{4}$, which completes the proof of this lemma.
\end{proof}

In the following, we give a precise formula for the number of distinct curves over $\F_q$, up to
isomorphism classes over $\overline{\F}_q$, of the family~\eqref{eq:Leg}.

\begin{theorem}
\label{thm:numj Leg} For any prime $p\ge 3$, for the number $\JL$ of distinct values of the
$j$-invariant of the family~\eqref{eq:Leg}, we have
$$
\JL=\displaystyle \fl{\frac{q+5}{6}}.
$$
\end{theorem}
\begin{proof}
We note that
$$\JL=\sum_{u \in \F_q \setminus \set{0,1}} \frac{1}{\#\cJ_{\LL,u}}.$$
Let $$N_k= \#\set{u: u \in \F_q\setminus \set{0,1} , ~\#\cJ_{\LL,u}=k}, \qquad k = 1, 2, \ldots .
$$
From Lemma~\ref{lem:JuL}, we see that $N_k=0$ for $k > 6$. Therefore,
\begin{equation}
\label{eq:JL and Nk} \JL=\sum_{k=1}^6 \frac{N_k}{k}.
\end{equation}

First, we let $p=3$. Form Lemma~\ref{lem:JuL}, we have $N_1=1$, $N_2=N_3=N_4=N_5=0$ and $N_6=q-3$.
Then, using~\eqref{eq:JL and Nk}, we obtain
$$\JL=(q+3)/6.$$

Now, we assume that $p>3$. From Lemma~\ref{lem:JuL}, we have $N_1=N_4=N_5=0$ and $N_3=3$.
Furthermore, $N_2$ is the number of roots of the polynomial $U^2-U+1$ in~$\F_q$. Since, $-3$, the
discriminant of this polynomial, is a quadratic residue in~$\F_q$ if and only if $q \equiv 1 \pmod
{3}$, we have $N_2=2$ if $q \equiv 1 \pmod {3}$ and $N_2=0$ if $q \equiv 2 \pmod {3}$. Because
$\sum_{k=1}^6 N_k=q-2$, we obtain $N_6=q-5-N_2$. Next, using~\eqref{eq:JL and Nk}, we have
$$\JL=
\left\{
\begin{array}{ll}
(q+5)/6,&   \text{ if }  q \equiv 1 \pmod {3}, \vspace{2pt}\\
(q+1)/6, &  \text{ if }  q \equiv 2 \pmod {3}.
\end{array}
\right.
$$
\end{proof}

%
%

Now, we give an exact formula for the number of $\F_q$-isomorphism classes of elliptic curves
over~$\F_q$ of the family~\eqref{eq:Leg}.

\begin{theorem}
\label{thm:numi Leg} For any prime $p\ge 3$, for the number $\IL$ of $\F_q$-isomorphism classes of
the family~\eqref{eq:Leg}, we have
$$\IL=
\left\{
\begin{array}{ll}
\displaystyle{\fl{\frac{7q+29}{24}}}&   \text{ if }  q \equiv 1 \pmod {12}, \vspace{2pt}\\
\displaystyle{\fl{\frac{q+2}{3}}} &  \text{ if }  q \equiv 3, 7 \pmod {12}, \vspace{2pt}\\
\displaystyle{\fl{\frac{7q+13}{24}}} &  \text{ if }  q \equiv 5,9 \pmod {12}, \vspace{2pt}\\
\displaystyle{ \;\; \frac{q-2}{3}} &  \text{ if }  q \equiv 11 \pmod {12}. \vspace{2pt}\\
\end{array}
\right.
$$
\end{theorem}
\begin{proof}
We note that $$\IL=\sum_{u \in \F_q \setminus \set{0,1}} \frac{1}{\#\cI_{\LL,u}}.$$ From
Lemma~\ref{lem:IuL}, we see that $1 \le \#\cI_{\LL,u} \le 6$.  Let
$$M_k= \#\set{u: u \in \F_q\setminus \set{0,1}, ~\#\cI_{\LL,u}=k},
\qquad k = 1, 2, \ldots, 6.
$$
Then,
\begin{equation}
\label{eq:ILMk} \IL=\sum_{k=1}^6 \frac{M_k}{k}.
\end{equation}

We partition $\F_q\setminus \set{0, 1}$ into the following sets:
$$\cA \cup  \cB,  $$
where as before $\cB$ is given by~\eqref{eq:BL} and
$$\cA=\set{u \in \F_q\setminus \set{0, 1} ~: u\ne -1,2,2^{-1}, u^2-u+1\ne 0 }.$$
Moreover, we write $$\cB= \cB_1 \cup  \cB_2,$$ where $$\cB_1=\set{u \in \F_q ~:~  u=-1,2,2^{-1}}
\mand  \cB_2=\set{u \in \F_q ~:~ u^2-u+1 = 0 }.$$ We note that the sets $\cB_1$ and $\cB_2$ are not
necessarily disjoint. In Table~\ref{Tab:LAB}, we show the cardinalities of above sets, where we let
$q \equiv r \pmod {3}$. For the case $q \equiv 0 \; (\text{mod } 3)$, we have
$\cB_1=\cB_2=\set{2}$.

\begin{table*}[ht]
\begin{center}
$\begin{array}{|c|c|c|c|c|c|} \hline \quad r\quad  & \;\#\cA \;& \; \#\cB_1\; & \;
\#\cB_2 \; \\
\hline
\hline \phantom{\int_{x_{x_1}}^{x^2}}\hspace{-20pt}  0  & \; q-3 \;& \; 1 \;&  \; 1 \; \\
\hline \phantom{\int_{x_{x_1}}^{x^2}}\hspace{-20pt}  1  & \; q-7 \;& \; 3 \;&  \; 2 \;\\
\hline \phantom{\int_{x_{x_1}}^{x^2}}\hspace{-20pt}  2  & \; q-5 \;& \; 3 \;&  \; 0 \;\\
   \hline
\end{array}$
\end{center}
\vspace{-1mm}
   \caption{Cardinalities of the sets $\cA, \cB_1, \cB_2$.}  \label{Tab:LAB}
\end{table*}

If $q \equiv 1 \; (\text{mod } 4)$, then $\chi_2(-1)=1$. In this case, we partition $\cA$ into the
following set
 $$\cA_1 \cup \cA_2 \cup \cA_3, $$
 where
\begin{eqnarray*}
\cA_1&=&\set{u \in \F_q\setminus \set{0, 1} ~:~u \notin \cB, \
\chi_2(u) = \chi_2(1-u)=-1 },\\
\cA_2&=&\set{u \in \F_q\setminus \set{0, 1}~:~u \notin \cB, \
\chi_2(u) \chi_2(1-u)=-1 },\\
\cA_3&=&\set{u \in \F_q\setminus \set{0, 1}~:~u \notin \cB, \
\chi_2(u) = \chi_2(1-u)=1 }.\\
\end{eqnarray*}
We note that, $\chi_2(2)=\chi_2(2^{-1})=1$ if and only if $q \equiv \pm 1 \pmod {8}$. Furthermore,
for all $u\in \cB_2$, we have $\chi_2(u)=\chi_2(1-u)=\chi_2(-1)$; see the proof
of~Lemma~\ref{lem:IuL}. Next, from Lemma~\ref{lem:Lchi}, we compute the the cardinalities of the
sets $\cA_1$, $\cA_2$, $\cA_3$, where $q \equiv 1 \; (\text{mod } 4)$; see Table~\ref{Tab:LAi},
where we let $q \equiv r \pmod {24}$.

\begin{table*}[ht]
\begin{center}
$\begin{array}{|c|c|c|c|c|c|}
\hline \quad r\quad  & \;\#\cA_1 \;& \; \#\cA_2\; & \;\#\cA_3 \; \\
\hline
\hline \phantom{\int_{x_{x_1}}^{x^2}}\hspace{-20pt} 1  & \; \frac{q-1}{4} \;& \;\frac{q-1}{2} \;&  \;\frac{q-5}{4}-5\; \\
\hline \phantom{\int_{x_{x_1}}^{x^2}}\hspace{-20pt} 5  & \; \frac{q-1}{4}-1 \;& \;\frac{q-1}{2}-2 \;&  \;\frac{q-5}{4}\; \\
\hline \phantom{\int_{x_{x_1}}^{x^2}}\hspace{-20pt} 9  & \; \frac{q-1}{4} \;& \;\frac{q-1}{2} \;&  \;\frac{q-5}{4}-1\; \\
\hline \phantom{\int_{x_{x_1}}^{x^2}}\hspace{-20pt} 13  & \; \frac{q-1}{4}-1 \;& \;\frac{q-1}{2}-2 \;&  \;\frac{q-5}{4}-2\; \\
\hline \phantom{\int_{x_{x_1}}^{x^2}}\hspace{-20pt} 17  & \; \frac{q-1}{4} \;& \;\frac{q-1}{2} \;&  \;\frac{q-5}{4}-3\; \\
\hline
\end{array}$
\end{center}
\vspace{-1mm}
   \caption{Cardinalities of the sets $\cA_1, \cA_2, \cA_3$, for $q \equiv 1 \pmod
{4}$}  \label{Tab:LAi}
\end{table*}

From Lemma~\ref{lem:IuL}, for all $u\in \F_q\setminus\set{0,1}$, we have
\begin{equation}
\label{equ:IL Ru} \#\cI_{\LL,u}= \left\{
\begin{array}{ll}
1, & \text{ if }  u = -1 \text{ and } p=3, \\
1, & \text{ if }  u = 2^{-1}, \ q \equiv 5 \pmod 8, \\
1, & \text{ if }  u\in \cB_2, \ q \equiv 7 \pmod {12},\\
2, & \text{ if }  u = -1, 2, \ q \equiv 5 \pmod 8,\\
2, & \text{ if }  u \in \cB_2,\ q \equiv 1 \pmod {12}  \text{ and } p>3,\\
2, & \text{ if }  u\in \cA_1 \text{ and } q \equiv 1 \pmod 4,\\
3, & \text{ if }  u = -1, 2, 2^{-1}, \ q \equiv 1,3,7 \pmod 8 \text{ and } p>3,\\
3, & \text{ if }  u\in \cA \text{ and } q \equiv 3 \pmod 4, \\
4, & \text{ if }  u\in \cA_2 \text{ and } q \equiv 1 \pmod 4,\\
6, & \text{ if }  u\in \cA_3 \text{ and } q \equiv 1 \pmod 4.
\end{array}
\right.
\end{equation}

Then, we obtain the values of $M_k$; see Table~\ref{Tab:LMk} (where $q \equiv r \pmod {24}$).
\begin{table*}[ht]
\begin{center}
$\begin{array}{|c|c|c|c|c|c|c|c|}
\hline \quad r\quad & \;M_1 \;& \;M_2 \;& \; M_3 \; & \; M_4 \; & \; M_5 \; & \; M_6 \;\\
\hline
\hline \phantom{\int_{x_{x_1}}^{x^2}}\hspace{-20pt} 1  & 0  &\; \frac{q+7}{4} \;& 3   &\; \frac{q-1}{2}  \; & 0  &\; \frac{q-25 }{4} \; \\
\hline \phantom{\int_{x_{x_1}}^{x^2}}\hspace{-20pt} 3  & 1  &\;  0            \;& q-3 &\;  0             \; & 0  &\; 0               \; \\
\hline \phantom{\int_{x_{x_1}}^{x^2}}\hspace{-20pt} 5  & 1  &\; \frac{q+3}{4} \;& 0   &\;  \frac{q-5}{2} \; & 0  &\; \frac{q-5}{ 4}  \; \\
\hline \phantom{\int_{x_{x_1}}^{x^2}}\hspace{-20pt} 7,19  & 2  &\;  0            \;& q-4 &\;  0             \; & 0  &\; 0               \; \\
\hline \phantom{\int_{x_{x_1}}^{x^2}}\hspace{-20pt} 9  & 1  &\; \frac{q-1}{4} \;& 0   &\;  \frac{q-1}{2} \; & 0  &\; \frac{q-9}{ 4}  \; \\
\hline \phantom{\int_{x_{x_1}}^{x^2}}\hspace{-20pt} 11,23 & 0  &\;  0            \;& q-2 &\;  0             \; & 0  &\; 0               \; \\
\hline \phantom{\int_{x_{x_1}}^{x^2}}\hspace{-20pt} 13 & 1  &\; \frac{q+11}{4} \;& 0  &\; \frac{q-5}{2}  \; & 0  &\; \frac{q-13 }{ 4} \; \\
\hline \phantom{\int_{x_{x_1}}^{x^2}}\hspace{-20pt} 17 & 0  &\; \frac{q-1}{4}  \;& 3  &\; \frac{q-1}{2}  \; & 0  &\; \frac{q-17 }{ 4} \; \\
\hline
\end{array}$
\end{center}
\vspace{-1mm}
   \caption{$M_k$, for $k=1,\ldots,6$.}  \label{Tab:LMk}
\end{table*}

Next, using~\eqref{eq:ILMk},  we compute:
$$
\IL= \left\{
\begin{array}{ll}
(7q+17)/24, &\text{ if } q \equiv 1 \pmod  {24}, \vspace{3pt}\\
q/3, &\text{ if } q \equiv 3 \pmod  {24}, \vspace{3pt}\\
(7q+13)/24,  &\text{ if } q \equiv 5 \pmod {24},\vspace{3pt}\\
(q+2)/3,  &\text{ if } q \equiv 7,19 \pmod {24},\vspace{3pt}\\
(7q+9)/24,  &\text{ if } q \equiv 9 \pmod {24},\vspace{3pt}\\
(q-2)/3,  &\text{ if } q \equiv 11,23 \pmod {24},\vspace{3pt}\\
(7q+29)/24,  &\text{ if } q \equiv 13 \pmod {24},\vspace{3pt}\\
(7q+1)/24,  &\text{ if } q \equiv 17 \pmod {24},
\end{array}
\right.
$$
which completes the proof.

\end{proof}

\section{Jacobi curves}
\label{sec: D-curves}

First, we consider the Jacobi qartic curves $\E_{\JQ,u}$ given by~\eqref{eq:JacQ} over a field~$\F$
with characteristic $p\ge 3$. Note that $u \ne \pm1$, since the curve $\E_{\JQ,u}$ is nonsingular.
The change of variable $(\bx,\by)\mapsto (\wX,\wY)$ defined by $ \wX=2(u+\frac{Y+1}{X^2})$ and
$\wY=\frac{2\wX}{X}, $ is a birational equivalence from $\E_{\JQ,u}$ to the elliptic curve
$\widetilde{\E}_{\JQ,u}$ defined by
\begin{equation*}
\label{eq:EC JacQ} \wY^2 = \wX^3-4u\wX^2 + 4(u^2-1)\wX
\end{equation*}
with $j$-invariant $j(\E_{\JQ,u}) = F(u)$ where
$$F(U) = \frac{64(U^2+3)^3}{(U^2-1)^2}.$$

\begin{lemma}\label{lem:JacQ}
For all $u\in \F$ with $u\ne \pm 1$, the Jacobi curve $\E_{\JQ,u}$ is birationally equivalent over
$\F$ to the Legendre curve $\E_{\LL, \frac{1-u}{2}}$.
\end{lemma}
\begin{proof}
We note that, the Jacobi-quartic curve $\E_{\JQ,u}$ is birationally equivalent over $\F$ to the
elliptic curve $\widetilde{\E}_{\JQ,u}$ defined by~\eqref{eq:EC JacQ}. Moreover, the map $(\wX,\wY)
\mapsto (\bx,\by)$ defined by
$$
\bx=\frac{\wX -2u+2}{4} \mand \by= \frac{-\wY}{8},
$$
is an isomorphism over $\F$ from the curve $\widetilde{\E}_{\JQ,u}$ to the Legendre curve~$\E_{\LL,
\frac{1-u}{2}}$.
\end{proof}

Now, we consider the curves $\E_{\JI,u}$ given by~\eqref{eq:JacI} over a field $\F$ with
characteristic~$p\ge 3$. We note that $u \ne 0,1$, since the curve $\E_{\JI,u}$ is nonsingular. The
change of variable $(\bx,\by,\bz)\mapsto (\wX,\wY)$ defined by $
\wX=\frac{u(\by-\bz)}{u\by-\bz+1-u}$ and $\wY=\frac{u(1-u)\bx}{u\by-\bz+1-u},$ is a birational
equivalence from $\E_{\JI,u}$ to the elliptic curve $\widetilde{\E}_{\JI,u}$ defined by
\begin{equation*}
\label{eq:EC JacI} \wY^2 = \wX^3-(u+1)\wX^2 + u\wX.
\end{equation*}
We note that $\widetilde{\E}_{\JI,u}=\E_{\LL,u}$. Moreover, the Jacobi intersection curve
$\E_{\JI,u}$ is birationally equivalent over $\F$ to the the Jacobi curve $\E_{\JQ,{1-2u}}$.


From Theorem~\ref{thm:numi Leg}, the following lemma gives the numbers of distinct curves over a
finite field~$\F_q$ of the families~\eqref{eq:JacQ} and~\eqref{eq:JacI}.

\begin{theorem}
\label{thm:numi Jacq} For any prime $p\ge 3$, for the numbers  $\IL$, $\IJQ$ and $\IJI$ of
$\F_q$-isomorphism classes of the families~~\eqref{eq:Leg}, \eqref{eq:JacQ} and \eqref{eq:JacI}
respectively, we have
$$
\IL=\IJQ=\IJI.
$$
\end{theorem}

\begin{proof}
From Lemma~\ref{lem:JacQ}, for all $u\in \F_q \setminus\set{0,1}$, the Legendre curve $\E_{\LL,u}$
in the family~\eqref{eq:Leg} is birationally equivalent to the curve $\E_{\JQ,1-2u}$ in the
family~\eqref{eq:JacQ}. Clearly, this correspondence is bijective. Moreover, they are birationally
equivalent to the curve $\E_{\JI,u}$ in the family~\eqref{eq:JacI}. Hence, the proof is complete.
\end{proof}

\begin{theorem}
\label{thm:numj Jacq} For any prime $p\ge 3$, for the numbers $\JJQ$ and~$\JJI$ of distinct values
of the $j$-invariant of the families~\eqref{eq:JacQ} and~\eqref{eq:JacI} respectively, we have
$$
\JJQ=\JJI=\displaystyle \fl{\frac{q+5}{6}}.
$$
\end{theorem}

\begin{proof}
This theorem is a direct consequence of Theorems~\ref{thm:numj Leg} and~\ref{thm:numi Jacq}.
\end{proof}

\section{Hessian curves}
\label{sec: H-curves}

We consider the curves $\E_{\HH,u}$ given by~\eqref{eq:Hes} over a finite field $\F_q$ of
characteristic~$p$. We note that $u^3 \ne 27$, since the curve $\E_{\HH,u}$ is nonsingular. For
$p>3$, the curve $\E_{\HH,u}$ is birationaly equivalent to the elliptic curve
\begin{equation}
\label{eq:H-W} \E_{\W,A_u, B_u}~:~ \wY^2 = \wX^3+A_u\wX +B_u
\end{equation}
where
$$A_u=-u(u^3 + 6^3)/3 \mand B_u=(u^6-540u^3-18^3)/27.$$
Furthermore, we have $j(\E_{\HH,u}) =  \(\frac{u(u^3+6^3)}{u^3-3^3}\)^3$. Therefore, the curve
$\E_{\HH,u}$ has the $j$-invariant $j(\E_{\HH,u}) = (F(u))^3$ where
\begin{equation}\label{eq:map Hes}
F(U) = \frac{U(U^3+216)}{U^3-27}.
\end{equation}

We consider the bivariate rational function $ F(U) - F(V) = {g(U,V)}/{l(U,V)} $ with two relatively
prime polynomials $g$ and $l$. We see that
\begin{equation}
\label{eq:H-guv} g(U,V)=(U-V)(UV-3U-3V-18)h(U,V),
\end{equation}
where $h(U,V)=(U^2+3U+9)V^2 + 3(U^2+12U -18)V+ 9(U^2-6U+36).$

For a fixed value $u\in \F_q$ with $u^3\ne 27$, let $g_u(V)=g(u,V)$. Next, for $u\in \F_q$ with
$u^3\ne 27$, we investigate the number of roots of the polynomial $g_u(V)=g(u,V)$. Moreover, for
$u\in \F_q $ with $u^3\ne 27$, we let
\begin{eqnarray*}
\cJ_{\HH,u}=\set{v~:~v \in \F_q, \ \ \E_{\HH,u} \cong_{\overline{\F}_q} \E_{\HH,v}},\;
\cI_{\HH,u}=\set{v~:~v \in \F_q, \ \ \E_{\HH,u} \cong_{\F_q} \E_{\HH,v}}.
\end{eqnarray*}


In the following, we give the cardinalities of $\cJ_{\HH,u}$ and $\cI_{\HH,u}$, for all $u\in \F_q$
with $u^3 \ne 27$.

\begin{lemma}\label{lem:Ju-Hes}
For all  $u\in \F_q$ with $u^3\ne 27 $, we have
$$
\#\cJ_{\HH,u}= \left\{
\begin{array}{ll}
1, & \text{ if }  p=2 \text{ and } u=0, \text{ or } p=3,\\
4, & \text{ if }  q \equiv 1 \pmod  3,\ p \ne 2 \text{ and } A_u = 0,\\
6, & \text{ if }  q \equiv 1 \pmod  3,\ p \ne 2 \text{ and } B_u = 0,\\
12, & \text{ if }  q \equiv 1 \pmod  3 \text{ and } A_uB_u\ne 0,\\
2, &\text{ if }   q \equiv 2  \pmod  3,\ p \ne 2 \text{ or } u \ne 0.
\end{array}
\right.
$$
\end{lemma}
\begin{proof}
For a fixed value $u\in \F_q$ with $u^3\ne 27$, let $g_u(V)=g(u,V)$ and
$$
\cZ_u=\set{v~:~ v \in \F_q, v^3\ne 27, \ g_u(v)=0 }.
$$
We note that, $v\in \cJ_u$ if and only if $F(u)^3=F(v)^3$, which
is equivalent to $F(u)=\zeta F(v)$ for some third root of unity $\zeta$ in $\F_q$. Moreover, we
have $F(\zeta v)=\zeta F(v)$ for all $\zeta \in \F_q$ with $\zeta^3=1$. Therefore, for all third
roots of unity $\zeta \in \F_q$, we have $v\in \cJ_u$ if and only if $\zeta v\in \cZ_u$. In other
words,
$$\cJ_{\HH,u}=\set{\zeta v: \  \zeta \in \F_q \text{ with } \zeta^3=1,  v \in \cZ_u}.$$
We note that, all third roots of unity in $\overline{\F}_q$ are in $\F_q$ if and only if $q \equiv
1 \pmod  3$. Otherwise, $\F_q$ has only unity as the trivial third root of unity. Moreover, if
$F(u)\ne 0$, i.e. $A_u\ne 0$, then for all elements $v_1, v_2 \in \cZ_u$ and for all distinct third
roots of unity $\zeta_1, \zeta_2 \in \F_q$, we see $\zeta_1 v_1 \ne \zeta_2 v_2$. If $F(u)=0$, then
for all $v \in \cZ_u$ and for all third roots of unity $\zeta \in \F_q$, we have $F(\zeta v)=0$,
so, $\zeta v \in \cZ_u$. Therefore,
\begin{equation}
\label{eq:ZuJuH} \#\cJ_{\HH,u}= \left\{
\begin{array}{ll}
\#\cZ_u &\text{ if } q \equiv 0,2  \pmod  3, \\
\#\cZ_u &\text{ if } q \equiv 1 \pmod  3 \text{ and } A_u = 0,\\
3\#\cZ_u, & \text{ if }  q \equiv 1 \pmod  3 \text{ and } A_u \ne 0,\\
\end{array}
\right.
\end{equation}
Here, we study the set $\cZ_u$ for all possibilities for $u$ and $q$.

If $p=3$, then we see $\cZ_u=\set{u}$ and $\cJ_{\HH,u}=\set{u}$. Next, we assume that $p\ne 3$. We
write
$$g_u(V)=-(u^3-27)(V-u)(V-v)h_u(V),$$ where $v=\frac{3(u+6)}{u-3}$ and $h_u(V)=V^2+\frac{\beta(u)}{\alpha(u)}V+\frac{\gamma(u)}{\alpha(u)}$, where
$$\alpha(U)=U^2+3U+9,\ \beta(U) =3(U^2+12U -18),\ \gamma(U) = 9(U^2-6U+36)$$ (see
Equation~\eqref{eq:H-guv}). We write,
$$h_u(V)=(V-v_1)(V-v_2),$$
where $v_i=\frac{3\zeta^i(u+6\zeta^i)}{u-3\zeta^i}$, for $i=1,2$, $\zeta \in \overline{\F}_q$ with
$\zeta^3=1$ and $\zeta \ne 1$.

For $p \ne 3$, the polynomial $g_u$ is square free if $\Delta_u$ the discriminant of $g_u$ is
nonzero. We have $\Delta_u=-3^{21}B_u^4$, so $g_u$ has distinct roots with multiplicity one if
$B_u\ne 0$.

We note that, if $v^3=27$ then~$g_u(v)=-3^5v(u^3-27)$. Thus, for all $v\in \cZ_u$, we have $v^3\ne
27$.

Now, we distinguish the following possibilities for $q$.
\begin{enumerate}
\item First, we assume that $p=2$. We have $\beta(u)=\gamma(u)=u^2$. For $u=0$, we have
$\cZ_0=\set{0}$ and $\cJ_{\HH,0}=\set{0}$. Next, we let $u\ne 0$, i.e., $A_u B_u\ne 0$.
We note that, the polynomial $h_u$ is irreducible over~$\F_q$ if and only if
$\Tr_{\F_q/\F_2}\(\frac{\gamma(u)/\alpha(u)}{(\beta(u)/\alpha(u))^2}\)=1$. Moreover,
$$\textstyle \Tr_{\F_q/\F_2}\(\frac{\alpha(u)}{\beta(u)}\)=\Tr_{\F_q/\F_2}(1+\frac{1}{u}+\frac{1}{u^2})=\Tr_{\F_q/\F_2}(1).$$
Furthermore, $\Tr_{\F_q/\F_2}(1)=1$ if and only if $q \equiv 2 \pmod  3$. Therefore, we have the
following cases.
\begin{itemize}
\item If $q \equiv 1 \pmod  3$, then $v_1, v_2 \in \F_q$. Hence,
$$\cZ_u=\set{u,v,v_1,v_2}.$$
\item
If $q \equiv 2 \pmod  3$, then $\cZ_u=\set{u,v}.$ 
\end{itemize}
\item
Second, we assume that $p > 3$. We distinguish the following possibilities for~$u\in \F_q$.
\begin{itemize}
\item
First, we assume that $B_u=0$. Then $A_u\ne 0$, since $u^3\ne 27$. Then, $u$ is a root of one of
the following equations.
\begin{equation}
\label{eq: terms of delta} U^2-6U-18=0  \quad \text{ or } \quad U^4+6U^3+54U^2-108U+324=0.
\end{equation}
Since $p>3$, the sets of solutions to the equations in \eqref{eq: terms of delta} do not intersect
and so may be considered separately.
\begin{itemize}
\item If $u^2-6u-18=0$, then $v=u$ and $v_1=v_2=-u+6$.
So, $$\cZ_u=\set{u, -u+6}.$$
\item If $u^4+6u^3+54u^2-108u+324=0$, then
$$g_u(V)=-(u^3-27)(V-u)^2(V-v)^2.$$
Furthermore, $v\ne u$. So, $$\cZ_u=\set{u,v}.$$
\end{itemize}
\item
Next, we assume that $B_u\ne 0$.
Let~$\D_u$ be the discriminant of $h_u$, i.e., $\D_u=-3^3(u^2-6u-18)^2/\alpha^2(u).$ We consider
the following cases for $q$.
\begin{itemize}
\item
If $q \equiv 1 \pmod  3$, then $\D_u$ is a quadratic residue in $\F_q$. Moreover, $\D_u\ne 0$. Thus, $h_u$ has two distinct roots $v_1, v_2$ in $\F_q$. Therefore, $$\cZ_u=\set{u,v,v_1,v_2}.$$ 
\item
If $q \equiv 2 \pmod  3$, then $\D_u$ is a quadratic non-residue in $\F_q$. Thus, $h_u$ has no root
in $\F_q$. Therefore,
$$\cZ_u=\set{u,v}.$$
\end{itemize}
\end{itemize}
\end{enumerate}
Therefore, we have
$$
\#\cZ_u= \left\{
\begin{array}{ll}
1, & \text{ if }  p=2 \text{ and } u=0, \text{ or } p=3,\\
2, & \text{ if }  q \equiv 1 \pmod  3,\ p \ne 2 \text{ and } B_u = 0,\\
4, & \text{ if }  q \equiv 1 \pmod  3 \text{ and } B_u\ne 0,\\
2, &\text{ if }   q \equiv 2  \pmod  3,\ p = 2 \text{ and } u\ne 0,\\
2, &\text{ if }   q \equiv 2  \pmod  3 \text{ and } p\ne 2.
\end{array}
\right.
$$
Next, using~\eqref{eq:ZuJuH}, we obtain the cardinality of $\cJ_{\HH,u}$.
\end{proof}

In the following Lemma, we study the $\F_q$-isomorphism classes of Hessian curves over~$\F_q$.

\begin{lemma}\label{lem:isoH}
For all elements $u,v \in \F_q$ with $u^3\ne 27, v^3\ne 27$, we have $\E_{\HH,u} \cong_{\F_q}
\E_{\HH,v}$  if and only if $u, v$ satisfy one of the following:
\begin{enumerate}
  \item $v=\zeta_1 u,$
  \item $v=\frac{3\zeta_1(u+6\zeta_2)}{u-3\zeta_2}$ and $q \equiv 1 \pmod  3$,
\end{enumerate}
for some third roots of unity $\zeta_1, \zeta_2 \in \F_q$.
\end{lemma}
\begin{proof}
Let $u,v \in \F_q$ with $u^3\ne 27, v^3\ne 27$. First, we assume that $\E_{\HH,u} \cong_{\F_q}
\E_{\HH,v}$. So, $v\in \cJ_{\HH,u}$. From the proof of Lemma~\ref{lem:Ju-Hes}, we see that
$$\cJ_{\HH,u}=\set{w \in \F_q : w=\zeta_1u \text{ or }\ w=\frac{3\zeta_1(u+6\zeta_2)}{u-3\zeta_2},  \zeta_1, \zeta_2 \in \overline{\F}_q, \zeta_1^3=1, \zeta_2^3=1  }.$$
We consider the following cases for $q$.
\begin{itemize}
\item If $p=3$, then we have $v=u$, which satisfy the property 1.
\item If $q \equiv 1 \pmod  3$, then $\F_q$ contains all third roots of unity in $\overline{\F}_q$. So, $u,v$ satisfy either the property 1 or 2.
\item If $q \equiv 2 \pmod  3$, then $\zeta=1$ is the only third root of unity in $\F_q$. From the proof of Lemma~\ref{lem:Ju-Hes}, we have
$$
\cJ_{\HH,u}= \left\{
\begin{array}{ll}
\set{u}, &   \text{ if }  p=2 \text{ and } u=0, \\
\set{u, \frac{3(u+6)}{u-3}}, & \text{ if }  q \equiv 2 \pmod  3, u^2-6u-18 \ne 0,\\
\set{u, -u+6}, & \text{ if }  q \equiv 2 \pmod  3 \text{ and } u^2-6u-18=0.\\
\end{array}
\right.
$$
Then, we distinguish the following cases to show that $v=u$.
\begin{itemize}
\item We let $p=2$. If $u=0$, then clearly $v=u$. So, we assume that $u\ne 0$.
Then, the map $(\bx,\by)\mapsto (\wX,\wY)$ defined by
$$
\wX=\frac{(u^3+1)(\bx+\by)}{u^3(\bx+\by+u)} \mand
\wY=\frac{(u^3+1)(\bx+(1+u^3)\by+u)}{u^6(\bx+\by+u)},
$$
is a birational equivalence from $\E_{\HH,u}$ to the elliptic curve
\begin{equation*}
\label{eq:BH-W} \E_{\BW, a_u, b_u }~:~ \wY^2 + \wX\wY =\wX^3 + a_u \wX^2 + b_u,
\end{equation*}
where $a_u=\frac{1}{u^3}$, $b_u=\(\frac{1+u^3}{u^4}\)^3$.
Similarly, $\E_{\HH,v}$ is birationally equivalent to $\E_{\BW,a_v,b_v}$.

If $v= \frac{3(u+6)}{u-3}$, i.e., $v= \frac{u}{u+1}$, then we have $b_u=b_v$. Next, we note that
the elliptic curves $\E_{\BW,a_u,b_u}$ and $\E_{\BW,a_v,b_v}$ are isomorphic over~$\F_q$ if and
only if there exists an element $\gamma$ in $\F_q$ such that $a_u+a_v=\gamma^2+\gamma$. In other
words, they are isomorphic if and only if $\Tr_{\F_q/\F_2}(a_u+a_v)=0$. But, for $v=
\frac{u}{u+1}$, we have
$$\textstyle\Tr_{\F_q/\F_2}(a_u+a_v)=\Tr_{\F_q/\F_2}(1+\frac{1}{u}+\frac{1}{u^2})=\Tr_{\F_q/\F_2}(1)=1.$$
So, the curves $\E_{\HH,u}$ and $\E_{\HH, \frac{u}{u+1}}$ are not isomorphic over $\F_q$. Hence, we
only have $v=u$.
\item We let $p>3$. Then, the map $(\bx,\by)\mapsto (\wX,\wY)$ given by
$$
\wX=Z - u^2  \mand \wY=3Z(Y-X),
$$
where $Z=\frac{4(u^3-27)}{3(u+3X+3Y)}$, is a birational equivalence from $\E_{\HH,u}$ to the
elliptic curve
\begin{equation*}
\label{eq:H-W} \E_{\W,A_u, B_u}~:~ \wY^2 = \wX^3+A_u\wX +B_u
\end{equation*}
where
$$A_u=\frac{-u(u^3 + 6^3)}{3} \mand B_u=\frac{u^6-540u^3-18^3}{27}.$$
Also, $\E_{\HH,v}$ is birationally equivalent to $\E_{\W,A_v,B_v}$. We note that,  $\E_{\W,A_u,B_u}
\cong_{\F_q} \E_{\W,A_v,B_v}$ if and only if there exist an element $\alpha \in \F^*_q$ such that
$A_v=\alpha^4 A_u$ and $B_v=\alpha^6 B_u$. We consider the following cases.
    \begin{itemize}
    \item If $u^2-6u-18 \ne 0$, we let $v=\frac{3(u+6)}{u-3}$. Then, $A_v=3^6A_u/(u-3)^4$ and $B_v=-3^9B_u/(u-3)^6$.
    We note that, $-3$ is a quadratic non-residue in $\F_q$. If $B_u=0$, then
    from~\eqref{eq: terms of delta}, we see $u^4 + 6u^3 + 54u^2 - 108u + 324=0$. Thus, $(u^2+3u+9)^2=-3^3(u-3)^2$, i.e., $-3$ is a quadratic residue in $\F_q$, which is a contradiction. So, we have $B_u\ne 0$. Therefore, $\E_{\HH,u}$ is not isomorphic over $\F_q$ to $\E_{\HH, \frac{3(u+6)}{u-3}}$. Hence, we have $v=u$.
    \item
    If $u^2-6u-18 =0$, we let $v=-u+6$. We note that, this case happens if $3$ is a quadratic residue in $\F_q$. For $q \equiv 2 \pmod  3$, $3$ is a quadratic residue in $\F_q$ if and only if $q \equiv 3 \pmod  {4}$. Let $u=3(1+\delta),$ where $\delta^2=3$. Then, $A_v=-(1-\delta)^4A_u/4$. Since, $-1$ is a quadratic non-residue in $\F_q$ with $q \equiv 3 \pmod  {4}$, the curves $\E_{\HH,u}$ and $\E_{\HH, -u+6}$ are not isomorphic over $\F_q$. Hence, we only have $v=u$.
    \end{itemize}
\end{itemize}
\end{itemize}
Now, we study the other direction of this Lemma. We consider the following cases for $u,v$.
\begin{itemize}
\item
First, we let $v=\zeta_1 u$, for some $\zeta_1 \in \F_q$ with $\zeta_1^3=1$. Then $\E_{\HH,u}$ can
be transformed to $\E_{\HH,v}$ via the invertible map $X \mapsto \zeta_1\wX$ and $Y \mapsto \wY$,
so $\E_{\HH,u} \cong_{\F_q} \E_{\HH,v}$.
\item
Second, we assume that $q \equiv 1 \pmod  3$. We let $v=\frac{3\zeta_1(u+6\zeta_2)}{u-3\zeta_2}$,
for some third roots of unity $\zeta_1, \zeta_2 \in \F_q$. Moreover, let $\zeta \in \F_q$ with
$\zeta^2+\zeta+1=0$. Then, $\E_{\HH,u}$ can be transformed to $\E_{\HH,v}$ via the invertible map
$(\bx,\by)\mapsto (\wX,\wY)$ defined by
$$
\wX=\frac{\zeta_1^2(-\zeta_2(\zeta + 1)X +\zeta Y +1)}{\zeta_2 X+Y+1} \mand \wY=\frac{\zeta_2\zeta
X -(\zeta +1)Y +1}{\zeta_2X+Y+1}.
$$
So, $\E_{\HH,u} \cong_{\F_q} \E_{\HH,v}$.
\end{itemize}
\end{proof}

Now, we study the cardinality of $\cI_{\HH,u}$, for $u\in \F_q$ with $u^3\ne 27$.

\begin{lemma}\label{lem:IuH}
For all  $u\in \F_q$ with $u^3\ne 27 $, we have
$$
\#\cI_{\HH,u}= \left\{
\begin{array}{ll}
1, &\text{ if }   q \equiv 0,2  \pmod  3,\\
4, & \text{ if }  q \equiv 1 \pmod  3,\ p \ne 2 \text{ and } A_u = 0,\\
6, & \text{ if }  q \equiv 1 \pmod  3,\ p \ne 2 \text{ and } B_u = 0,\\
12, & \text{ if }  q \equiv 1 \pmod  3 \text{ and } A_uB_u\ne 0,\\
\end{array}
\right.
$$
\end{lemma}
\begin{proof}
Let $u\in \F_q$ with $u^3\ne 27 $. From Lemma~\ref{lem:isoH}, we see that
$$
\cI_{\HH,u}= \left\{
\begin{array}{ll}
\set{u},     & \text{ if }  q \equiv 0,2  \pmod  3,\\
\cJ_{\HH,u}, & \text{ if }  q \equiv 1 \pmod  3.
\end{array}
\right.
$$
Then Lemma~\ref{lem:Ju-Hes} completes the proof.
\end{proof}

The following theorems give the number of distinct Hessian curves over $\F_q$.

\begin{theorem}
\label{thm:numj Hes} For any prime $p$, for the number $\JH$ of distinct values of the
$j$-invariant of the family~\eqref{eq:Hes}, we have
$$
\JH= \left\{
\begin{array}{ll}
\displaystyle{\; q-1} &   \text{ if }  q \equiv 0 \pmod {3}, \vspace{2pt}\\
\displaystyle\fl{\frac{q+11}{12}}&   \text{ if }  q \equiv 1 \pmod {3}, \vspace{2pt}\\
\displaystyle \fl{\frac{\;q\;}{2}} &   \text{ if }  q \equiv 2 \pmod {3}.
\end{array}
\right.
$$
\end{theorem}

\begin{proof}
We note that $$\JH=\sum_{u \in \F_q, u^3\ne 27 } \frac{1}{\#\cJ_{\HH,u}}.$$ Let $$N_k= \#\set{u: u
\in \F_q, u^3\ne 27, ~\#\cJ_{\HH,u}=k}, \qquad k = 1, 2, \ldots .
$$
From Lemma~\ref{lem:Ju-Hes}, we see that $N_k=0$ for $k > 12$. Then,
\begin{equation}
\label{eq:JH-Nk} \JH=\sum_{k=1}^{12} \frac{N_k}{k}.
\end{equation}

We consider the following possibilities for $q$.
First, we assume that $q \equiv 0\pmod 3$. From Lemma~\ref{lem:Ju-Hes}, we see $N_1=q-1$ and
$N_k=0$, for $k \ge 2$. Then, using~\eqref{eq:JH-Nk}, we have $$\displaystyle \JH=q-1,\; \text{ if
} q \equiv 0\pmod 3.$$

Second, we assume that $q \equiv 1\pmod 3$. 
In this case, the number of third roots of 27 in $\F_q$ equals 3. If $p=2$, then $N_1=1$,
$N_{12}=q-4$ and $N_k=0$, for $2 \le k \le 11$ (see~Lemma~\ref{lem:Ju-Hes}).
Using~\eqref{eq:JH-Nk}, we have $\JH=(q+8)/12$.

If $p\ne 2$, then $N_k=0$ for $k\ne 4,6,12$. Also, $N_4=d_1$, $N_6=d_2$ and $N_{12}=q-3-d_1-d_2$,
where $d_1$, $d_2$ are the numbers of $u\in \F_q$ with $A_u=0$, $B_u=0$ respectively (see
Lemma~\ref{lem:Ju-Hes}). Then, using~\eqref{eq:JH-Nk}, we have
$$\JH=(q+2d_1+d_2-3)/12.$$

We see that, $d_1$ is the number of solutions in $\F_q$ to the equation $U(U^3+6^3)=0$. Then,
$d_1=4$, since $q \equiv 1\pmod 3$ and $p\ne 2$. Next, we compute $d_2$,  i.e., the number of
solutions in $\F_q$ to the equations in~\eqref{eq: terms of delta}. We note that, the sets of
solutions to the equations in \eqref{eq: terms of delta} are disjoint. Moreover, they do not
intersect with the set of the third roots of 27 in $\F_q$.

The discriminant of the quadratic equation  $U^2-6U-18=0$ in \eqref{eq: terms of delta} equals~108.
So, it has two solutions if 3 is a quadratic residue in $\F_q$ or no solution if 3 is a quadratic
non-residue in $\F_q$. We note that $-3$ is a quadratic residue in $\F_q$, since $q \equiv 1\pmod
3$. Also, $-1$ is a quadratic residue in $\F_q$ if and only if $q \equiv 1 \pmod 4$. Hence, the
number of solutions to the quadratic equation  in \eqref{eq: terms of delta} is 2 if $q \equiv
1\pmod {12}$ and 0 if $q \equiv 7\pmod {12}$.

The quartic equation in \eqref{eq: terms of delta} can be factored as
$$(U^2+3(1+\omega)U+9(1-\omega))(U^2+3(1-\omega) U+9(1+\omega)),$$
where $\pm \omega$, are the quadratic roots of $-3$ in $\F_q$. Since, the discriminants of above
quadratic factors are $27(1\pm \omega)^2$, the number of solutions to the quadratic equation in
\eqref{eq: terms of delta} is 4 if $q \equiv 1\pmod {12}$ and 0 if $q \equiv 7\pmod {12}$.

Therefore, $d_2=6$ if $q \equiv 1\pmod {12}$ and $d_2=0$ if $q \equiv 7\pmod {12}$. Hence,
\begin{equation*}
\JH= \left\{
\begin{array}{ll}
(q+11)/12,&   \text{ if }  q \equiv 1 \pmod {12}, \vspace{3pt}\\
(q+8)/12,&   \text{ if }  q \equiv 4 \pmod {12}, \vspace{3pt}\\
(q+5)/12, &   \text{ if }  q \equiv 7 \pmod {12}.
\end{array}
\right.
\end{equation*}

Now, we assume that $q \equiv 2 \pmod {3}$. Then, the number of third roots of 27 in $\F_q$ equals
1. If $p=2$, then Lemma~\ref{lem:Ju-Hes} shows that $N_1=1$, $N_2=q-2$ and $N_k=0$, for $k \ge 3$.
Then, using~\eqref{eq:JH-Nk}, we have $\JH=q/2$. If $p\ne 2$, then from Lemma~\ref{lem:Ju-Hes}, we
have $N_1=0$, $N_2=q-1$ and $N_k=0$, for $k \ge 3$. Next, using~\eqref{eq:JH-Nk}, we obtain
$\JH=(q-1)/2$. Hence,
\begin{equation*}
\JH= \left\{
\begin{array}{ll}
q/2,&   \text{ if }  q \equiv 2 \pmod {6}, \vspace{3pt}\\
(q-1)/2,&   \text{ if }  q \equiv 5 \pmod {6}.
\end{array}
\right.
\end{equation*}
\end{proof}

In the following theorem, we study the number of $\F_q$-isomorphism classes of Hessian curves.

\begin{theorem}
\label{thm:numi Hes} For any prime $p$, for the number $\IH$ of $\F_q$-isomorphism classes of the
family~\eqref{eq:Hes}, we have
$$
\IH= \left\{
\begin{array}{ll}
\displaystyle\fl{ \frac{q+11}{12}}&   \text{ if }  q \equiv 1 \pmod {3}, \vspace{2pt}\\
\displaystyle{\;\; {q-1}} &   \text{ if }  q \equiv 0,2 \pmod {3}.
\end{array}
\right.
$$
\end{theorem}
\begin{proof}
From Lemma~\ref{lem:IuH}, we recall that
$$
\#\cI_{\HH,u}= \left\{
\begin{array}{ll}
1,     & \text{ if }  q \equiv 0,2  \pmod  3,\\
\#\cJ_{\HH,u}, & \text{ if }  q \equiv 1 \pmod  3.
\end{array}
\right.
$$
Since $\#\IH=\sum_{u \in \F_q, u^3\ne 27 } \frac{1}{\#\cI_{\HH,u}}$, if $q \equiv 0,2  \pmod  3$,
we see $\#\IH=q-1$. Moreover, if $q \equiv 0,2  \pmod  3$,  Theorem~\ref{thm:numj Hes} gives the
cardinality of $\IH$. Therefore, we have
$$
\IH= \left\{
\begin{array}{ll}
(q+11)/12,&   \text{ if }  q \equiv 1 \pmod {12}, \vspace{2pt}\\
(q+8)/12,&   \text{ if }  q \equiv 4 \pmod {12}, \vspace{2pt}\\
(q+5)/12,&   \text{ if }  q \equiv 7 \pmod {12}, \vspace{2pt}\\
q-1, &   \text{ if }  q \equiv 0,2 \pmod {3}.
\end{array}
\right.
$$
%
\end{proof}

\section{Generalized Hessian curves}

We consider the \emph{generalized Hessian} curves $\E_{\GH,u,v}$ given by~\eqref{eq:GH} over a
finite field $\F_q$ of characteristic~$p$. Obviously, a Hessian curve $\E_{\HH,u}$ given
by~\eqref{eq:Hes} is a generalized Hessian curve $\E_{\GH,u,v}$ with $v=1$. Moreover, the curve
$\E_{\GH,u,v}$ over $\F_q$, via the map $(X,Y) \mapsto (\wX,\wY)$ defined by
\begin{equation}\label{eq:G-H}
  \wX=X/\zeta  \mand \wY=Y/ \zeta
\end{equation}
with $\zeta^3=v$, is isomorphic over $\overline{\F}_q$ to the Hessian curve
$\E_{\HH,\frac{u}{\zeta}}~:~ \wX^3+\wY^3+1=\frac{u}{\zeta} \wX \wY.$ Therefore, for the
$j$-invariant of $\E_{\GH,u,v}$, we have
\begin{equation}\label{eq:jGH}
  j(\E_{\GH,u,v})=j(\E_{\HH,\frac{u}{\zeta}})=
  \frac{1}{v}\left(\frac{u(u^3+216v)}{u^3-27v}\right)^3\enspace.
\end{equation}

%

%
%
%
%

%

For a fixed element $v$ in $\F_q$, with $v\ne 0$, we let
\begin{equation*}
  \JGHv =\set{j \mid j=j(\E_{\GH,u,v}),\ u\in \F_q, u^3\ne 27v },
\end{equation*}
and let $\SJGH=\bigcup_{v\in \F^*_q} \JGHv\ $. Clearly, we have $\JGH=\#\SJGH$.

\begin{lemma}\label{lem:jinv}
  Let $v_1, v_2 \in \F^*_q$ and let $v=v_1/v_2$. If $v$ is a cube in
  $\F_q$, then we have $\cJ_{\mathrm{GH}, v_1}= \cJ_{\mathrm{GH}, v_2}$, otherwise
  we have $\cJ_{\mathrm{GH}, v_1} \cap \cJ_{\mathrm{GH}, v_2} =
  \set{0}$.
\end{lemma}

\begin{proof}
  Suppose $v=\zeta^3$ is a cube in $\F_q$. For all $u\in
  \F_q$ with $u^3\ne 27v$, we have
  $j(\E_{\GH,u,v_1})=j(\E_{\GH,u/\zeta,v_2})$ and similarly
  $j(\E_{\GH,u,v_2})=j(\E_{\GH,\zeta u,v_1})$. Therefore, $\cJ_{\mathrm{GH}, v_1}= \cJ_{\mathrm{GH}, v_2}$.

  Next, suppose that $v$ is not a cube in $\F_q$. Let $j\in \cJ_{\mathrm{GH}, v_1} \cap \cJ_{\mathrm{GH}, v_2}$. Then,
  $$j=\frac{1}{v_1}\left(\frac{u_1({u_1}^3+216v_1)}{{u_1}^3-27v_1}\right)^3
  = \frac{1}{v_2}\left(\frac{u_2({u_2}^3+216v_2)}{{u_2}^3-27v_2}\right)^3,$$
  for some $u_1,u_2\in \F_q$. If $j\ne 0$, we see that $v=v_1/v_2$ is
  a cube in $\F_q$, which is a contradiction. So, $\cJ_{\mathrm{GH}, v_1} \cap \cJ_{\mathrm{GH}, v_2} = \set{0}$.
\end{proof}

\begin{lemma}\label{lem:JHv}
  For $q \equiv 1 \pmod 3$, if $v$ is not a cube in $\F_q$, we
  have $\#\JGHv=(q+2)/3$.
\end{lemma}
\begin{proof}
  For $u\in \F_q$ with $u^3\ne 27v$, we let $j(\E_{\GH,u,v}) =
  \frac{1}{v}\ (F(u))^3$ where $F(U)=\frac{U(U^3+216 v)}{U^3-27 v}$.
  We consider the bivariate rational function $F(U)-F(V)$. We obtain
$$ F(U) - F(V)=\frac{U-V}{U^3-27v}\ \prod_{i=1}^3 \left(U-\frac{3\zeta_i(V+6\zeta_i)}{V-3\zeta_i}
\right)\thinspace, $$
  where, $\zeta_1,\zeta_2,\zeta_3$ are three cubic roots of $v$ in
  $\overline{\F}_q$. For all $u_1, u_2 \in \F_q$ with $u_1^3\ne 27v$, $u_2^3\ne
  27v$, we see that $F(u_1)=F(u_2)$ if and only if $u_1=u_2$. Hence, $F$ is
  an injective map over $\F_q$ and we have $F(\F_q)=\F_q$. Now, consider the
  map $\kappa: \F_q^* \rightarrow \F_q^*$ by
  $\kappa(x)=\frac{1}{v}x^3$. This map is $3:1$, if $q \equiv 1
  \pmod 3$. Therefore, $\#\JGHv=(q-1)/3+1$.
\end{proof}

\begin{theorem}\label{thm:numj GHes}
For any prime $p$, for the number $\JGH$ of distinct values of the $j$-invariant of the
family~\eqref{eq:GH}, we have
$$
  \JGH=
  \begin{cases}
    q-1, &   \text{if $q \equiv 0 \pmod {3}$}\\
    \fl{ (3q+1)/4},&   \text{if $q \equiv 1 \pmod {3}$}\\
    \fl{q/2}, &   \text{if $q \equiv 2 \pmod {3}$}
  \end{cases}\enspace.
$$
\end{theorem}

\begin{proof}
  If $q \not\equiv 1 \pmod 3$, every element of $\F_q$ is a cube in
  $\F_q$. Next, Lemma~\ref{lem:jinv} implies that, for all $v\in
  \F^*_q$, we have $\JGHv=\cJ_{\GH, 1}$. Therefore, $\JGH=\#\cJ_{\GH,1}$. Then,
  from Theorem \ref{thm:numj Hes}, we have
$$
  \JGH=
  \begin{cases}
    q-1, &   \text{if $q \equiv 0 \pmod {3}$}\\
    \fl{q/2}, & \text{if $q \equiv 2 \pmod {3}$}
  \end{cases}\enspace.
$$
  For $q \equiv 1 \pmod 3$, we fix a value $v \in \F_q$ that is not a
  cube in $\F_q$. Following Lemma~\ref{lem:jinv}, we write
  $\SJGH=\JGHv \cup \cJ_{\GH, v^2} \cup \cJ_{\GH, 1} $, where $\JGHv
  \cap \cJ_{\GH, v^2} = \JGHv \cap \cJ_{\GH, 1} = \cJ_{\GH, v^2} \cap
  \cJ_{\GH, 1}=\set{0}$. By Lemma~\ref{lem:JHv}, we have $\#\JGHv=
  \#\cJ_{\GH, v^2}=(q+2)/3$.  Moreover, from Theorem \ref{thm:numj Hes}, we
  have
$$
  \# \cJ_{\GH, 1}=
  \begin{cases}
    (q+11)/{12},& \text{if $q \equiv 1 \pmod {12}$}\\
    (q+8)/{12},&   \text{if $q \equiv 4 \pmod {12}$}\\
    (q+5)/{12}, &   \text{if $q \equiv 7 \pmod {12}$}
  \end{cases}\enspace.
$$
  Therefore, we have\vspace{-\baselineskip}
$$
  \JGH=
  \begin{cases}
    (3q+1)/{4},&   \text{if $q \equiv 1 \pmod {12}$}\\
    3q/{4},&   \text{if $q \equiv 4 \pmod {12}$}\\
    (3q-1)/{4}, &   \text{if $q \equiv 7 \pmod {12}$}
  \end{cases}\thinspace,
$$
  which completes the proof.
\end{proof}


We recall from Theorem \ref{thm:numi Hes} that the number of $\F_q$-isomorphism classes of Hessian
curves over~$\F_q$ is $\fl{(q+11)/12}$ if $q \equiv 1 \pmod {3}$ and $q-1$ if $q \not\equiv 1 \pmod
{3}$. The following theorem gives explicit formulas for the number of distinct generalized Hessian
curves, up to $\F_q$-isomorphism, over the finite field $\F_q$.

\begin{theorem}\label{thm:numi GHes}
For any prime $p$, for the number $\IGH$ of $\F_q$-isomorphism classes of the family~\eqref{eq:GH},
we have
$$
  \begin{cases}
    \fl{(3(q+3)/{4}},& \text{if $q \equiv 1 \pmod {3}$}\\
    \;\;{q-1}, & \text{if $q \equiv 0,2 \pmod {3}$}
  \end{cases}\enspace.
$$
\end{theorem}

\begin{proof}  If $q \equiv 0,2 \pmod 3$, then every generalized Hessian curve is
  $\F_q$-isomorphic to a Hessian curve via the map given by Equations
  \eqref{eq:G-H}. So, $\IGH$ equals the number of $\F_q$-isomorphism
  classes of the family of Hessian curves over~$\F_q$. Then, from
  Theorem \ref{thm:numi Hes}, we have $\IGH=q-1$ if $q \not\equiv 1
  \pmod 3$.

  Next, suppose that $q \equiv 1 \pmod 3$. For $a\in \F_q$, let
  $i_{\GH}(a)$ be the set of $\F_q$\nobreakdash-isomorphism classes of
  generalized Hessian curves $\GH_{u,v}$ with $j(\E_{\GH,u,v})=a$. So,
  $\#i_{\GH}(a)$ is the number of distinct generalized Hessian curves with
  $j$-invariant $a$ that are twists of each other. Clearly,
  $\#i_{\GH}(a)=0$, if $a \not\in \SJGH$. We note that, for all elliptic
  curve $E$ over $\F_q$, we have $\#E(\F_q)+\#E_t(\F_q)=2q+2$, where
  $E_t$ is the nontrivial quadratic twist of $E$. We also recall that
  the order of the group of $\F_q$-rational points of a generalized
  Hessian curve is divisible by $3$ (see \cite[Theorem 2]{FJ}).
  Since $q \equiv 1 \pmod 3$, if the isomorphism class of $\E_{\GH, u,v}$
  is in $i_{\GH}(a)$ then the isomorphism class of the nontrivial
  quadratic twist of $\E_{\GH, u,v}$ is not in $i_{\GH}(a)$. So, $\#i_{\GH}(a)=1$
  if $a\in \SJGH$ and $a\ne 0,1728$. Moreover, one can show that
  $\#i_{\GH}(a)=3$ if $a=0$ and $\#i_{\GH}(a)=1$ if $a=1728$, $a\ne 0$ and $a
  \in \SJGH$. Therefore, we have
$$
  \IGH=\sum_{a\in \F_q} i_{\GH}(a)=\sum_{a\in \SJGH} i_{\GH}(a)=2+ \sum_{a\in
    \SJGH} 1=2+\JGH\enspace.
$$
  From the proof of Theorem~\ref{thm:numj GHes}, we have
$$
  \IGH=
  \begin{cases}
    (3q+9)/{4},&   \text{if $q \equiv 1 \pmod {12}$}\\
    (3q+8)/{4},&   \text{if $q \equiv 4 \pmod {12}$} \\
    (3q+7)/{4}, &   \text{if $q \equiv 7 \pmod {12}$}
  \end{cases}\thinspace,
$$
  which completes the proof.
\end{proof}

\section{Comments and Open Questions}
There are also several more interesting families of Elliptic curves such as Montgomery curves
~\cite{Mo}, Edwards curves \cite{Edw}, and its variants \cite{BerLan1,BBLP,BLRF}. The numbers of
distinct $j$-invariants of the families of Edwards curves in \cite{Edw,BerLan1} have been studied
in~\cite[Theorms 3 and 5]{FS}. Moreover, the explicit formulas for the numbers of
$\F_q$-isomorphism classes of these families are given in ~\cite[Theorems 5,6 and 8]{Fa}. The
proofs of the latter Theorems can be provided by our method. However, we refer to ~\cite{FMW} for
the other proofs via different techniques.

For future work, we plan to study the exact formulas for the number of distinct elliptic curves $E$
over $\F_q$, where $E(\F_q)$ has a specific small subgroup. In particular, we give the explicit
formulas for the number of elliptic curves $E$ over $\F_q$ with a point of small order $n$.

\vspace{5mm}

\noindent \textbf{Acknowledgment.} The author would like to thank Igor Shparlinski for his interest
and support of this work. 

\end {document}